\documentclass[12pt]{amsart}
\usepackage{verbatim}
\usepackage{amssymb}
\usepackage{enumerate}

\newif\ifdeveloping
\newtheorem{theorem}{Theorem}[section]
\newtheorem{proposition}[theorem]{Proposition}
\newtheorem{lemma}[theorem]{Lemma}
\newtheorem{corollary}[theorem]{Corollary}

\newtheorem{obs}[theorem]{Observation}
\newtheorem{qtheorem}{Theorem\makebox[-1mm]{}}  
  
\newtheorem{qproblem}{Problem\makebox[-1mm]{}}  
\newtheorem{btheorem}{Black Box Theorem}

\theoremstyle{definition}

\newtheorem{definition}[theorem]{Definition}

\theoremstyle{remark}
  
{}

\ifdeveloping
\usepackage[notref,notcite]{showkeys}
\fi

\newcommand{\prtime}{{\count0=\time\divide\count0 by 60
\count1=-\count0\multiply\count1 by 60
\advance\count1 by \time
\the\count0:\the\count1}
}

\def\myheads#1;#2;{
\pagestyle{myheadings}
\markboth{{\sc\hfill #1\hfill\protect\makebox[0cm][r]{\rm\today; \prtime}}}
{{\sc\protect\makebox[0cm][l]{\rm\today;\ \prtime}\hfill #2\hfill}}
\thispagestyle{myheadings}
}

\newcommand{\acal}{{\mathcal A}}
\newcommand{\bcal}{{\mathcal B}}

\newcommand{\dcal}{{\mathcal D}}
\newcommand{\ecal}{{\mathcal E}}

\newcommand{\gcal}{{\mathcal G}}
\newcommand{\hcal}{{\mathcal H}}

\newcommand{\pcal}{{\mathcal P}}

\newcommand{\qcal}{{\mathcal Q}}

\newcommand{\tcal}{{\mathcal T}}

\newcommand{\abar}{\dot{A}}
\newcommand{\bbar}{\dot{B}}

\newcommand{\supp}{\operatorname{supp}}

\newcommand{\setm}{\setminus}
\newcommand{\empt}{\emptyset}
\newcommand{\subs}{\subset}
\newcommand{\sups}{\supset}
\newcommand{\oo}{{{\omega}_1}}
\newcommand{\dom}{\operatorname{dom}}
\newcommand{\ran}{\operatorname{ran}}
\def\<{\left\langle}
\def\>{\right\rangle}

\def\br#1;#2;{{\bigl[ {#1} \bigr]}^ {#2} }
\def\brsmall#1;#2;{{[ {#1} ]}^ {#2} }

\newcommand{\force}{\Vdash}

\def\to{\longrightarrow}

\newcommand{\oot}{{\omega}_2}
\newcommand{\nar}{{\not\makebox[-7pt]{}\rightarrow}}

\theoremstyle{plain}

\def\finf#1;#2;#3;{\operatorname{Fn}_{#1}(#2,#3)}

\def\finmok{\finf m;\oo;K;} 

\def\domm#1[#2]{\dom(#2)}
\def\domd{dom-disjoint}

\newcommand{\Homo}[2]{\operatorname{Hom}(#1,#2)}

\newcommand{\fink}[1]{\operatorname{\dcal}^{(k)}(#1)}
\newcommand{\seqk}[2]{\operatorname{\mathbb D}^{(k)}_{#1}(#2)}
\newcommand{\dvec}{\vec D}
\newcommand{\evec}{\vec E}
\newcommand{\nvec}{\vec N}
\newcommand{\mvec}{\vec M}

\newcommand{\ark}{AR$^{(k)}$}

\newcommand{\dplus}{\stackrel{.}+}

\begin{document}

\author[L. Soukup]{
Lajos Soukup}
\address{Alfr{\'e}d R{\'e}nyi Institute of Mathematics\\Budapest, V. 
Re\'altanoda utca 13-15\\H-1053, Hungary }  
\email{soukup@renyi.hu}

\title{Indestructible colourings and rainbow Ramsey theorems}
\thanks{The preparation of this paper was supported by the 
Hungarian National Foundation for Scientific Research grant no. 61600
and 68262}

\begin{abstract}
We show that if a colouring
$c$ establishes $\oot\nar [(\oo;{\omega})]^2_{\omega}$
then $c$ establishes this negative partition relation
in each Cohen-generic extension of the ground model, i.e.
this property of $c$ is Cohen-indestructible.
This result yields a negative answer to a  question of Erd\H os and Hajnal:
it is consistent that GCH holds and there is a colouring
$c:\br {\oot};2;\to 2$ establishing
$\oot\nar [(\oo;{\omega})]^2_2$
such that some colouring $g:\br \oo;2;\to 2$
can not be embedded into $c$. 

It is also consistent that $2^{\oo}$ is arbitrarily large,
and there is a function $g$ establishing
$2^{\oo}\nar[(\oo,\oot)]^2_{\oo}$ but there is no  
uncountable $g$-rainbow subset of $2^{\oo}$.
 
We also show that if GCH holds then 
for each $k\in {\omega}$ there is 
a $k$-bounded colouring $f:\br \oo;2;\rightarrow \oo$ and 
there are two
c.c.c posets $\pcal$ and $\qcal$ such that
\begin{displaymath}
V^{\pcal}\models \text{``$f$ c.c.c-indestructibly establishes 
$\oo\nar^* [(\oo;\oo)]_{k-bdd}$'',}
\end{displaymath}
but
\begin{displaymath}
V^{\qcal}\models \text{``
$\oo$ is the union of countably many
$f$-rainbow sets
''.} 
\end{displaymath}
\end{abstract}

\maketitle 
\ifdeveloping 
\myheads{Indestructible Colouring};{Indestructible Colouring};
\fi

\section{Introduction}

Erd\H os and Hajnal  observed, in \cite{EH1},    that 
if  a graph $G$ establishes ${\omega}_1\nar \bigl(({\omega},\oo)\bigr)^2_2$ then 
$G$ is universal for countable graphs, i.e., every countable graph
is isomorphic to a spanned subgraph of $G$.
This result can not be generalized for higher cardinals because
of the following result of Shelah {\cite[Theorem 4.1]{Sh}}:
{\em (a) Assume that  ${\kappa}$, ${\lambda}$ and ${\tau}$ are cardinals
of cofinality greater than ${\omega}$ and $G$ is a graph on 
${\kappa}$. Then the property
\begin{itemize}
\item[($*$)]  {\em $G$ establishes
 ${\kappa}\nar (({\lambda},{\tau}))^2_2$}  
\end{itemize}
can not be destroyed by adding a single Cohen real, i.e.
if $V\models (*)$ then $V^{Fn({\omega},2)}\models (*)$.}\\
(b) {\em If you add a Cohen reals to some model $V$ then
in the generic extension 
there is a graph $C$ on $\oo$ 
which is not isomorphic to any spanned subgraph
of some graph $G$ from $V$.}

Learning this result
Erd\H os and Hajnal  raised the following
question in {\cite[Problem 6.b]{EH2}}:
{\em Assume that a graph $G$ establishes 
${\omega}_2\nar (\oo\dplus{\omega})^2_2$. 
Do all graphs of cardinality $\aleph_1$ embed into 
$G$?}

We answer their question in the negative  in theorem
\ref{tm:ellenpelda}. The proof is based on theorem \ref{tm:cind}
which says that  the property ``{\em $G$ establishes 
$\oot\nar ((\oo;{\omega}))^2_2$}'' 
is indestructible by adding  arbitrary numbers of Cohen reals
to the ground model.

Given a colouring $f:\br X;n;\to C$ a subset $P\subs X$ is called
{\em rainbow for $f$} 
(or {\em $f$-rainbow})
iff $f\restriction \br P;n;$ is one-to-one.
We also answer  another question of  Hajnal, \cite[Problem 4.1]{H2},      
in  the negative in theorem \ref{tm:con}: it is consistent 
with GCH that there is a function   $f$  which establishes
$\oot\nar [(\oo;{\omega})]^2_{\oo}$ such that  there is 
 no uncountable $f$-rainbow set.

In theorem \ref{tm:bau} we show that 
it is also consistent that $2^{\oo}$ is arbitrarily large,
and a function $g$ establishes
$2^{\oo}\nar[(\oo,\oot)]^2_{\oo}$ such that there is no  
uncountable $g$-rainbow set.

In the second part of the paper we deal with rainbow Ramsey theorems
concerning ``bounded'' functions.
A function $f:\br X;n;\rightarrow C$  is {\em ${\mu}$-bounded}
iff {$|f^{-1}\{c\}|\le {\mu}$} for each $c\in C $.

Let us recall some ``arrow'' notations:\\
{${\lambda}\rightarrow^* ({\alpha})^n_{{\kappa}-{\rm bdd}}$} holds
iff for every ${\kappa}$-bounded colouring of $\br {\lambda};n;$
there is a rainbow set of order type ${\alpha}$,\\
{\em ${\lambda}\rightarrow^* [({\alpha};{\beta})]_{{\kappa}-{\rm bdd}}$}
holds
iff for every ${\kappa}$-bounded colouring $c$ of $\br {\lambda};2;$
there is a set $A\subs {\lambda}$ of order type ${\alpha}$ and there is a set 
$B\subs {\lambda}$ of order type ${\beta}$ such that $\sup A\le \sup B$ and  
$|[A; B]\cap c^{-1}\{{\xi}\}|<{\kappa}$ for each ${\xi}\in \ran c$,
where $[A;B]=\{\{{\alpha},{\beta}\}:{\alpha}\in A, {\beta}\in B, {\alpha}<{\beta}\}$.

Clearly ${\lambda}\rightarrow^* ({\alpha})^2_{{\kappa}-{\rm bdd}}$
implies ${\lambda}\rightarrow^* [({\alpha};{\alpha})]_{{\kappa}-{\rm bdd}}$.

We say that  a function $f$  {\em c.c.c-indestrictibly establishes 
the negative partition  relation
$\Phi\nar^*\Psi$}
iff 
\begin{displaymath}
\text{
$V^P\models$ ``$f$ establishes 
$\Phi\nar^* \Psi$
''
}
\end{displaymath}
for each c.c.c poset $P$.

Since $\oo\rightarrow({\alpha})^2_{2}$ holds for ${\alpha}<\oo$
by \cite{BH}, and
 it was proved by Galvin, \cite{G}, that
 ${\lambda}\rightarrow ({\alpha})^n_k$ implies 
${\lambda}\rightarrow^* ({\alpha})^n_{k-{\rm bdd}}$
, we have 
$\oo\rightarrow^*({\alpha})^2_{2-bdd}$ for ${\alpha}<\oo$.
Moreover, Galvin, \cite{G},  showed that 
\begin{qtheorem}
CH implies that $\oo\nar^* (\oo)^2_{2-bdd}$.
\end{qtheorem}
On the other hand, Todorcevic, \cite{T}, proved that 
\begin{qtheorem}
PFA  implies that $\oo\rightarrow^* (\oo)^2_{2-bdd}$.
\end{qtheorem}

Abraham,  Cummings and Smyth showed that
$MA_{\aleph_{1}}$ is not enough to get  
 $\oo\rightarrow^* (\oo)^2_{2-bdd}$. 
More precisely, they proved the following theorem:
\begin{qtheorem}[{\cite[Theorem 3]{ACS}}]
It is consistent that there is a function $c:\br \oo;2;\to \oo$ which 
c.c.c-indestructibly 
establishes
 $\oo\nar^* (\oo)^2_{2-bdd}$.
\end{qtheorem}

They also showed that the property 
``{\em $c$ establishes $\oo\nar^* (\oo)^2_{2-bdd}$}'' 
is not automatically c.c.c-indestructible:
\begin{qtheorem}[{\cite[Theorem 4]{ACS}}]
If CH holds and there is a Suslin-tree then  
there is a function $c':\br \oo;2;\to 2$ and there is a c.c.c poset 
$P$ such that  
\begin{enumerate}[(a)]
\item  $c'$  establishes $\oo\nar^* (\oo)^2_{2-bdd}$,
\item $V^P\models$ there is an uncountable $c'$-rainbow set.
\end{enumerate}
\end{qtheorem}
We show that even the negative partition
relation 
$\oo\nar^* [(\oo;\oo)]_{k-bdd}$ is
consistent  with $MA_{\aleph_{1}}$ for each 
$k\in {\omega}$. 

Moreover, Abraham and Cumming used two different functions 
in their theorems above. 
We show that a single function can play double role.

\begin{theorem}\label{tm:acgen}
If GCH holds then 
for each $k\in {\omega}$ there is 
a $k$-bounded colouring $f:\br \oo;2;\rightarrow \oo$ and 
there are two
c.c.c posets $\pcal$ and $\qcal$ such that
\begin{displaymath}
V^{\pcal}\models \text{``$f$ c.c.c-indestructibly establishes 
$\oo\nar^* [(\oo;\oo)]_{k-bdd}$'',}
\end{displaymath}
but
\begin{displaymath}
V^{\qcal}\models \text{``
$\oo$ is the union of countably many
$f$-rainbow sets
''.} 
\end{displaymath}
\end{theorem}

\section{On a  problem of Erd\H os and Hajnal.}

To formulate our results we need to introduce some notations.
Given two functions $f:\br X;2;\to C$ and $d:\br Y;2;\to C$ we say
that {\em $d$ can be embedded into $f$},
($d\Rightarrow f$, in short), iff there is a one-to-one map 
$\Phi:Y\to X$ such that $d(\{y,y'\})=f(\{\Phi(y),\Phi(y')\})$
for each $\{y,y'\}\in \br Y;2;$.

Hajnal, \cite{H}, proved that 
it is consistent with GCH that there is a colouring
establishing ${\omega}_2\nar (\oo\dplus
{\omega})^2_2$. As it turns out, his argument gives following stronger result:
\begin{proposition}\label{pr:con}
It is consistent that GCH holds and there is a function
$f:\br {\oot};2;\to \oo$ establishing
$\oot\nar [(\oo;{\omega})]^2_{\oo}$.
\end{proposition}

Since Hajnal's proof was never published we sketch his argument.

\begin{proof}[Proof of Proposition \ref{pr:con}]
Define a poset $\pcal=\<P,\le\>$ as follows.
The underlying set $P$ consists of triples 
$\<c,\acal,{\xi}\>$ where  $c:\br \supp (c);2;\to {\omega}$
for some $\supp (c)\in \br \oot;{\omega};$, 
$\acal\subs \br \supp(c);{\omega};$ is a countable family and ${\xi}\in \oo$.

Put $\<d,\bcal,{\zeta}\>\le \<c,\acal,{\xi}\>$ iff 
\begin{enumerate}[(P1)]
 \item $c\subset d$, $\acal\subset \bcal$, ${\xi}\le {\zeta}$, 
  \item  for each $A\in\acal$ and 
for each ${\beta}\in (\supp  (d)\setm \supp (c))\cap \min A$
\begin{displaymath}
{\xi}\subs d''[\{{\beta}\},A].
\end{displaymath}
\end{enumerate}
Then $\pcal$ is a ${\sigma}$-complete, ${\omega}_2$-c.c. poset
and if $\gcal$ is the generic filter for $\pcal$
then $g=\cup\{c:\<c,\acal,{\xi}\>\in\gcal\}$
establishes $\oot\nar [(\oo;{\omega})]^2_{\oo}$
in $V[\gcal]$.
\end{proof}

Proposition \ref{pr:con} validates 
the following question of 
Erd\H os and Hajnal, {\cite[Problem 6.b]{EH2}}:
{\em Assume that a graph $G$ establishes 
${\omega}_2\nar (\oo\dplus{\omega})^2_2$. 
Do all graphs of cardinality $\aleph_1$ embed into 
$G$?}

To answer this question in the negative 
we prove a preservation theorem which makes us possible to apply 
Shelah's method  from \cite[theorem 4.1]{Sh}.

\begin{theorem}\label{tm:cind}
If ${\mu}\le \oo$
and $c$ establishes $\oot\nar [(\oo;{\omega})]^2_{\mu}$ then
$V^{Fn({\kappa},2)}\models$ 
``$c$ establishes $\oot\nar [(\oo;{\omega})]^2_{\mu}$. 
\end{theorem}

\begin{proof}

The following lemma is straightforward.
\begin{lemma}\label{lm:equiv}
Let ${\mu}\le \oo$ and $c:\br \oot;2;\to {\mu}$.
The followings are equivalent:
\begin{enumerate}[(1)]
\item $c$ establishes $\oot\nar [(\oo;{\omega})]^2_{\mu}$,
\item $\forall B\in \br \oot;{\omega};$ $\forall {\nu}\in {\mu}$
  \begin{displaymath}
   |\bigr\{{\alpha}<\min B: {\nu}\notin c''[\{{\alpha}\}, B]\bigl\}|\le {\omega},
  \end{displaymath}
\item
$\forall \bcal\in \br {\brsmall \oot;{\omega};};{\omega};$
$\forall {\nu}\in {\mu}$
  \begin{displaymath}
   |\bigl\{{\alpha}<\min \cup\bcal:
\exists B\in \bcal\ {\nu}\notin 
c''[\{{\alpha}\}, B] \bigr\}|\le {\omega}. 
  \end{displaymath}
\end{enumerate}
\end{lemma}

Assume on the contrary that the theorem fails.
We can assume that we add just $\oo$ many Cohen reals to $V$, 
i.e. ${\kappa}=\oo$.
We can choose
 ${\xi}\in \oot$, ${\nu}\in {\mu}$, $p\in Fn(\oo,2)$ and names
$\abar$ and $\bbar$ such that 
\begin{displaymath}
p\force \abar\in \br {\xi};\oo;\land  \bbar\in \br \oot\setm {\xi};{\omega};
\land {\nu}\notin c''[\abar,\bbar].
\end{displaymath}
  
We can assume that $\bbar\in V^{Fn({\omega},2)}$ and $\dom p\subs {\omega}$.
For each $q\in Fn({\omega},2)$ with $q\le p$ put
\begin{displaymath}
B(q)=\{{\zeta}:\exists r\in Fn({\omega},2)\ r\le q \land 
r\force {\zeta}\in \bbar\}.  
\end{displaymath}
Let $\bcal=\{B(q): q\in Fn({\omega},2), q\le p\}$ and  
$A'=\{{\alpha}\in \oot:\exists r\le p\ r\force {\alpha}\in \abar\}$.
Then $A'\in \br {\xi};\oo;$ and 
$\bcal\in \br {\brsmall \oot\setm {\xi};{\omega};};{\omega};$.
Hence, by lemma \ref{lm:equiv},  there is ${\alpha}\in A'$ such that 
${\nu}\in c''[\{{\alpha}\}, B(q)]$ for each 
$q\in Fn({\omega},2)$ with $q\le p$.
Pick $s\in Fn(\oo,2)$ with $s\force {\alpha}\in \abar$.
Then ${\nu}\in c''[\{{\alpha}\}, B(s\restriction {\omega})]$, i.e.
there is ${\beta}\in \oot \setm {\xi}$ and $r\in Fn({\omega},2)$
such that $r\le s\restriction {\omega}$ and $r\force {\beta}\in \bbar$.
Then 
\begin{displaymath}
s\cup r\force {\alpha}\in \abar\land {\beta}\in \bbar \land
{\nu}\notin c''[\abar,\bbar],  
\end{displaymath}
but $c({\alpha},{\beta})={\nu}$. Contradiction.
\end{proof}

\begin{theorem}\label{tm:ellenpelda}
For $2\le {\mu}\le \oo$
it is consistent that GCH holds and there is a colouring
$f:\br {\oot};2;\to {\mu}$ establishing
$\oot\nar [(\oo;{\omega})]^2_{\mu}$
such that $g\not\Rightarrow f$ for some colouring $g:\br \oo;2;\to 2$.
  \end{theorem}

\begin{proof}
By proposition \ref{pr:con} we can assume that in the ground model
GCH holds and there is a function    
$f:\br {\oot};2;\to \oo$ establishing
$\oot\nar [(\oo;{\omega})]^2_{\oo}$.

If ${\mu}\le \oo$ and ${\pi}_{\mu}:\oo\rightarrow {\mu}$ is onto
then $f_{\mu}={\pi}_{\mu}\circ f$
establishes $\oot\nar [(\oo;{\omega})]^2_{\mu}$.

Then, by \cite[Theorem 4.1]{Sh}, in $V^{Fn({\omega},2)}$
there is a function $d:\br \oo;2;\rightarrow 2$ such that
$d\not\Rightarrow f_{\mu}$.

Since
\begin{displaymath}
\text{$V^{Fn({\omega},2)}\models$  
{\em  $f_{\mu}$
establishes $\oot\nar [(\oo;{\omega})]^2_{\mu}$}}
\end{displaymath}  
 by theorem \ref{tm:cind}, we are done.
\end{proof}

As it was observed by Hajnal, the construction of theorem
\ref{tm:ellenpelda} above left open 
the following question which he raised in \cite[Problem 4.1]{H2}:
\begin{qproblem}
Assume  GCH holds and a colouring
$c:\br {\oot};2;\to \oo$ establishes
$\oot\nar [(\oo;{\omega})]^2_{\oo}$.
Does there exist a $c$-rainbow set of size $\oo$?
\end{qproblem}

Before answering this question let us recall some
positive results of  Hajnal. In  \cite{H2}, he proved  that  
\begin{qtheorem}
(1) If $f: \br \oo;2; \to \oo$ establishes 
$\oo\nar [({\omega},\oo)]^2_{\oo} $ 
then $d \Rightarrow f$ for each  $d:\br {\omega};2;\to \oo$.\\
(2)  If $f :\br \oo;2;\to{\omega}$ establishes
  $\oo\nar[(\oo,\oo)]^2_{\omega}$ then there exists an infinite 
$f$-rainbow  set.
\end{qtheorem}

When we colour the pairs of $\oo$
we can not expect uncountable rainbow sets 
because of the following fact.
\begin{proposition}\label{pr:ch}
If CH holds then  there is a function
$f:\br {\oo};2;\to \oo$ 
such that 
\begin{enumerate}[(1)]
\item  $f$ establishes
$\oo\nar [({\omega};\oo)]^2_{\oo}$,
\item there is no uncountable $f$-rainbow.
\end{enumerate}
\end{proposition}

\begin{proof}[Proof of proposition \ref{pr:ch}]
Enumerate $\br \oo;{\omega};$ as $\{A_{\alpha}:{\omega}\le
{\alpha}<\oo\}$
such that $A_{\alpha}\subs {\alpha}$.
By induction on ${\alpha}$, ${\omega}\le {\alpha}<\oo$,  define
 $f({\xi},{\alpha})$ for ${\xi}<{\alpha}$ such that 
\begin{enumerate}
  \item ${\alpha}\subset \{f({\xi},{\alpha}):{\xi}\in A_{\beta}\}$ for
  ${\beta}<{\alpha}$,
\item $A_{\beta}\cup \{{\alpha}\}$ is not an $f$-rainbow for
  ${\beta}<{\alpha}$.
\end{enumerate}
Then $f$ satisfies (1) and (2).
\end{proof}

Next we answer \cite[Problem 4.1]{H2}  in  the negative.
\begin{theorem}\label{tm:con}
It is consistent that GCH holds and there is a function
$g:\br {\oot};2;\to \oo$ 
such that 
\begin{enumerate}[(1)]
\item  $g$ establishes
$\oot\nar [(\oo;{\omega})]^2_{\oo}$,
\item there is no uncountable $g$-rainbow.
\end{enumerate}
\end{theorem}

\begin{proof}[Proof of theorem \ref{tm:con}]
The naive approach is to try to modify
the order of the poset $P$ from the proof of proposition 
\ref{pr:con} 
by adding a condition (P3) to the definition of the order:
\begin{enumerate}[(P1)] \addtocounter{enumi}{2}
  \item  for each $A\in\acal$ and 
for each ${\beta}\in (\supp  (d)\setm \supp (c))$
the set $A\cup\{{\beta}\}$ is not a $d$-rainbow. 
\end{enumerate}
Unfortunately this approach does not work because
the modified poset does not satisfies $\oot$-c.c.
 
So we will argue in a different way.

Define the poset $P$ as follows.
The underlying set $P$ consists of quadruples 
$\<c,\acal,{\xi},\dcal\>$ where
\begin{enumerate}[(i)]
  \item $c:\br \supp (c);2;\to {\omega}$
for some $\supp (c)\in \br \oot;{\omega};$, 
\item $\acal\subs \br \supp(c);{\omega};$ is a countable family,
\item  ${\omega}\le {\xi}< \oo$,
\item  $\dcal\subs \br \supp(c);{\omega};\times \oo$ is a countable family,
\item   $\forall \<D,{\sigma}\>\in \dcal$ 
$(\forall {\gamma}\in \supp(c))$
$|\{{\delta}\in D:c({\gamma},{\delta})< {\sigma}\}|={\omega}$.
\end{enumerate}

Put $\<d,\bcal,{\zeta},\ecal\>\le \<c,\acal,{\xi},\dcal\>$ iff
\begin{enumerate}[(a)]
  \item $c\subset d$, $\acal\subset\bcal$, ${\xi}\le {\zeta}$,
$\dcal\subs \ecal$,
\item for each $A\in\acal$ and 
for each ${\beta}\in (\supp  (d)\setm \supp (c))\cap \min A$
\begin{displaymath}
{\xi}\subs d''[\{{\beta}\}, A].
\end{displaymath}
\end{enumerate}
 
Clearly $\le $ is  a partial order on $P$
and  $\pcal=\<P,\le\>$ is ${\sigma}$-complete.

\begin{lemma}
$\pcal$ is ${\omega}_2$-c.c.  
\end{lemma}

\begin{proof}[Proof of the lemma]
We say  that two conditions, $p=\<c,\acal,{\xi},\dcal\>$
and $p'=\<c',\acal',{\xi}',\dcal'\>$, are {\em twins}
iff there is an order preserving  bijection 
${\varphi}:\supp (c)\to \supp (c')$ such that 
\begin{enumerate}[(1)]
  \item $K=\supp (c)\cap \supp (c')$ is an initial segment of
both $\supp (c)$ and $\supp (c')$,
\item $K<\supp (c)\setm K<\supp (c')\setm K$,
\item $c({\xi},{\eta})=c'({\varphi}({\xi}),{\varphi}({\eta}))$ for
  each $\{{\xi},{\eta}\}\in \br \supp(c);2;$,
\item $\acal'=\{{\varphi}''A:A\in \acal\}$,
\item ${\xi}={\xi}'$,
\item $\dcal'=\{\<{\varphi}''D,{\sigma}\>:\<D,{\sigma}\>\in\dcal\}$. 
\end{enumerate}
It is enough to show that if
$p$ and $p'$ are twins then 
they have  a common extension 
$q=\<d,\bcal,{\rho},\ecal\>$.
Let $\bcal=\acal\cup\acal'$, ${\rho}={\xi}={\xi}'$
and $\ecal=\dcal\cup\dcal'$.

We should define $d({\nu},{\mu})$ for ${\nu}\in \supp (c)\setm K$
and ${\mu}\in\supp (c')\setm K$.

We enumerate all ``tasks'' as follows:
Let 
\begin{displaymath}
  \tcal_0=\{\<{\beta},A',{\zeta}\>:{\beta}\in \supp (c)\setm K,A'\in\acal',
A'\subs \supp (c')\setm K,{\zeta}<{\xi}'\},
\end{displaymath}
\begin{multline}\notag
  \tcal_1=\{\<{\gamma},\<D',{\sigma}'\>,n\>:
{\gamma}\in \supp (c)\setm
  K,\\ \<D',{\sigma}'\>\in\dcal'\setm \dcal,
|D'\setm K|={\omega}, n<{\omega}\}
\end{multline}
and
\begin{multline}\notag
  \tcal_2=\{\<{\gamma}',\<D,{\sigma}\>,n\>:
{\gamma}'\in \supp (c')\setm
  K,\\ \<D,{\sigma}\>\in\dcal\setm \dcal',
|D\setm K|={\omega}, n<{\omega}\}.
\end{multline}
Since $\tcal=\tcal_0\cup\tcal_1\cup \tcal_2$
is countable we can pick pairwise distinct 
ordinals $\{{\eta}_x:x\in \tcal\}$ such that 
\begin{enumerate}[(a)]
  \item if $x=\<{\beta},A',{\zeta}\>\in \tcal_0$ then ${\eta}_x\in A'$,
\item if $x=\<{\gamma},\<D',{\sigma}'\>,n\>\in \tcal_1$ 
then ${\eta}_x\in D'\setm  K$,
\item if $x=\<{\gamma}',\<D,{\sigma}\>,n\>\in\tcal_2$ 
then ${\eta}_x\in D\setm  K$.
\end{enumerate}
Choose a function $d:\br \supp (c)\cup\supp (c');2;\to \oo$
such that 
\begin{enumerate}[(1)]
  \item $d\supset c\cup c'$,
\item  $d({\beta},{\eta}_x)={\zeta}$ for $x=\<{\beta},A',{\zeta}\>\in
  \tcal_0$,
\item $d({\gamma},{\eta}_x)=0$ for  
$x=\<{\gamma},\<D',{\sigma}'\>,n\>\in \tcal_1$,
\item $d({\gamma}',{\eta}_x)=0$
for $x=\<{\gamma}',\<D,{\sigma}\>,n\>\in\tcal_2$.
\end{enumerate}

Let $q=\<d,\bcal,{\eta},\ecal\>$.
To show $q\in P$ we should check only condition (v).
So let $\<D,{\sigma}\>\in\ecal$ and 
${\gamma}\in \supp(d)$. Assume that 
$\<D,{\sigma}\>\in \dcal$.  (The case $\<D,{\sigma}\>\in \dcal'$
is similar.) 

If ${\gamma}\in \supp (c)$
then $d\restriction[\{{\gamma}\},D]=c\restriction[\{{\gamma}\},D]$  
so we are done.
So we can assume that ${\gamma}\in\supp (c')\setm K$.

If $D\setm K$ is finite then the set
\begin{displaymath}
  E=\{{\delta}\in D\cap K: c({\delta},{\varphi}^{-1}({\gamma}))<{\sigma}\}
\end{displaymath}
is infinite because $p\in P$ satisfies (v)
and for each ${\delta}\in E$ we have $d({\delta},{\gamma})=
c'({\delta},{\gamma})=c'({\varphi}({\delta}),{\gamma})=
c({\delta},{\varphi}^{-1}({\gamma}))<{\sigma}$.
So we can assume that $D\setm K$ is infinite.

In this case $x_n=\<{\gamma},\<D,{\sigma}\>,n\>\in \tcal_2$ for
$n\in {\omega}$,
so $d({\gamma},{\eta}_{x_n})=0<{\sigma}$ and 
$\{{\eta}_{x_n}:n\in {\omega}\}\in \br D;{\omega};$.

So $q\in P$.

It is straightforward that $q\le p$ because no instances of (b)
should be checked.

Finally we verify $q\le p'$.
Since condition (a) is clear, assume that $A'\in \acal'$
and ${\beta}\in \supp (c)\setm K$ with ${\beta}<\min A'$.
Since $\sup K<{\beta}$ we have $A'\subs \supp(c')\setm K$.
Hence for each ${\zeta}<{\xi}$ we have
$x=\<{\beta},A',{\zeta}\>\in\tcal_0$
so $d({\beta},{\eta}_x)={\zeta}$. Thus 
${\xi}\subs d''[\{{\beta}\},A']$.

This completes the proof of the lemma.  
\end{proof}

Let $\gcal$ be  the generic filter for $\pcal$
and put $g=\cup\{c:\<c,\acal,{\xi}\>\in\gcal\}$.

{\noindent \bf Claim:}
{\em 
$g$ establishes $\oot\nar [(\oo;{\omega})]^2_{\oo}$
in $V[\gcal]$.
}

Indeed,  let 
$p=\<c,\acal,{\xi},\dcal\>\in P$.
If $A\in \br \supp (c);{\omega}; $ 
and ${\eta}\in \oo$
then $p'=\<c,\acal\cup\{A\},\max ({\xi},{\eta}),\dcal\>\le
p$ and
for each $ {\beta}\in \min A \setm \supp (c)$ 
\begin{displaymath}
  p'\force 
{\eta}\subs g''[\{{\beta}\},A].
\end{displaymath}

{\noindent \bf Claim:}
{\em There is no uncountable $g$-rainbow 
set in $V[\gcal]$.}

Indeed, assume that $p_0 \force \dot X\in \br \oot;\oo;$.
Since $\pcal$ is ${\sigma}$-complete there are 
$p\le p_0$, 
$p=\<c,\acal,{\xi},\dcal\>$,
 and  $D\in \br \supp(c);{\omega};$
such that $p\force D\subs \dot X$.
Let
$p'=\<c,\acal,{\xi},
\dcal\cup\{\<D,(\sup \ran (c))+1\>\}.
\>$.
Then $p'\in P$ and $p'\le p$. Moreover 
\begin{displaymath}
p'\force \text{$\dot X$ is not a $g$-rainbow}. 
\end{displaymath}
Indeed, work in $V[\gcal]$, where $p'\in\gcal$. 
Write $X=\{{\xi}_{\nu}:{\nu}\in \oo\}$.
Then for each ${\nu}<{\omega}$ there is 
${\gamma}_{\nu}<\sup \ran (c)+1$ and ${\delta}_{\nu}\in D$ with 
$g({\delta}_{\nu}, {\xi}_{\nu})={\gamma}_{\nu}$.
Then there are ${\nu}<{\mu}<\oo$ with ${\gamma}_{\nu}={\gamma}_{\mu}$.
Then $g({\delta}_{\nu}, {\xi}_{\nu})={\gamma}_{\nu}={\gamma}_{\mu}=
g({\delta}_{\mu}, {\xi}_{\mu})$ and ${\xi}_{\nu}\ne {\xi}_{\mu}$, 
i.e. $X$ is not a $g$-rainbow.

\medskip
So, by the claims above, $g$ satisfies the requirements of the
theorem. 
\end{proof}

Baumgartner proved that if $CH$ holds, $P=Fn(\br
{\kappa};2;,\oo;\oo)$ for some cardinal 
${\kappa}\ge \oot$, and $\gcal$ is the generic filter above $P$,
then the function $g=\cup\gcal$ establishes 
$\oot\nar [(\oo,\oot)]^2_{\oo}$.
We prove a related result here.

\begin{theorem}\label{tm:bau}
If CH holds and ${\kappa}\ge \oot$ is a cardinal then 
there is a ${\sigma}$-complete, ${\omega}_2$-c.c. poset 
$P$ such that in $V^P$  there is a function
$g:\br {\kappa};2;\to \oo$ 
such that 
\begin{enumerate}[(1)]
\item  $g$ establishes
${\kappa}\nar [(\oo,\oot)]^2_{\oo}$.
\item there is no uncountable $g$-rainbow subset of ${\kappa}$.
\end{enumerate}
\end{theorem}

\begin{proof}
  Define the poset $P$ as follows.
The underlying set $P$ consists of pairs 
$\<c,\dcal\>$ where
\begin{enumerate}[(i)]
  \item $c:\br \supp (c);2;\to {\omega}$
for some $\supp (c)\in \br {\kappa};{\omega};$, 
\item  $\dcal\subs \br \supp(c);{\omega};\times \oo$ is a countable family,
\item   $\forall \<D,{\sigma}\>\in \dcal$ 
$(\forall {\gamma}\in \supp(c))$
$|\{{\delta}\in D:c({\gamma},{\delta})< {\sigma}\}|={\omega}$.
\end{enumerate}

Put $\<d,\ecal\>\le \<c,\dcal\>$ iff
 $c\subset d$ and
$\dcal\subs \ecal$.

Then 
 $\le $ is  a partial order, and 
 $\pcal$ is ${\sigma}$-complete.

We say  that two conditions, $p=\<c,\dcal\>$
and $p'=\<c',\dcal'\>$, are {\em twins} iff
 there is an order preserving  bijection 
${\varphi}:\supp (c)\to \supp (c')$ such that 
\begin{enumerate}[(1)]
\item ${\varphi}({\xi})={\xi}$ for ${\xi}\in \supp (c)\cap \supp (c')$, 
\item $c({\xi},{\eta})=c'({\varphi}({\xi}),{\varphi}({\eta}))$ for
  each $\{{\xi},{\eta}\}\in \br \supp(c);2;$,
\item $\dcal'=\{\<{\varphi}''D,{\sigma}\>:\<D,{\sigma}\>\in\dcal\}$. 
\end{enumerate}

\begin{lemma}\label{lm:twins}
Assume that $p=\<c,\dcal\>$, $p'=\<c',\dcal'\>$ are twins. 
Let $q\le p$, $q=\<d,\ecal\>$, such that 
$\supp (d)\cap \supp (c')=\supp (c)\cap \supp (c')$.
Let $A\in \br \supp (d)\setm \supp (c');{\omega};$, 
${\xi}\in \supp (c') \setm \supp (c)$ and ${\rho}<\oo$. 
Then there is a common extension 
$r=\<c_r,\dcal_r\>$ of $q$ and $p'$ such that 
${\rho}\subs c_r''[\{{\xi}\},A]$.   
\end{lemma}

\begin{proof}[Proof of the lemma]
Write 
$K=\supp (c)\cap \supp (c')$ 
and fix the function ${\varphi}$ witnessing 
that $p$ and $p'$ are twins.
  Let 
\begin{displaymath}
  \tcal_0={\rho},
\end{displaymath}
 \begin{displaymath}
 \tcal_1=\{\<{\gamma},\<D',{\sigma}'\>,n\>:{\gamma}\in \supp(d)\setm K,
\<D',{\sigma}'\>\in\dcal', |D'\setm K|={\omega},
n\in {\omega}\},
\end{displaymath}
 \begin{displaymath}
 \tcal_2=\{\<{\gamma}',\<E,{\sigma}\>,n\>:{\gamma}'\in \supp
 c'\setm K,
\<E,{\sigma}\>\in\ecal, |E\setm K|={\omega},
n\in {\omega}\}.
\end{displaymath}

Since $\tcal=\tcal_0\cup \tcal_1\cup \tcal_2$
is countable we can pick pairwise distinct 
ordinals $\{{\eta}_x:x\in \tcal\}$ such that 
\begin{enumerate}[(a)]
\item if $x={\chi}\in {\rho}$ then ${\eta}_x\in A$,
\item if $x=\<{\gamma},\<D,{\sigma}\>,n\>\in \tcal_1\cup \tcal_2$ 
then ${\eta}_x\in D\setm  K$.
\end{enumerate}
Let $c_r\sups d\cup c_{\nu}$ such that 
\begin{enumerate}[(i)]
  \item $c_r({\eta}_x,{\xi})={\chi}$ if $x={\chi}\in \tcal_0$,
  \item $c_r({\eta}_x,{\gamma})=0$ if 
$\<{\gamma},\<D,{\sigma}\>,n\>\in\tcal_0\cup\tcal_1$.
\end{enumerate}
To prove $r=\<c_r,\dcal'\cup\ecal\>\in P$
it is enough to check condition (iii).

Assume first that 
$\<D,{\sigma}\>\in \dcal'$. 

If ${\gamma}\in \supp (c')$
then $c_r\restriction[\{{\gamma}\},D]=c'\restriction[\{{\gamma}\},D]$  
so we are done.
So we can assume that ${\gamma}\in\supp(d)\setm K$.

If $D\setm K$ is finite then 
$\<{\varphi}^{-1}D,{\sigma}\>\in\dcal\subs \ecal$,
and $D\cap K= {\varphi}^{-1}D\cap K$, so
the set
\begin{displaymath}
  H=\{{\delta}\in D\cap  K: d({\delta},{\gamma})<{\sigma}\}
\end{displaymath}
is infinite because $q\in P$ satisfies (iii),
and $H\subs\{{\delta}\in D: c_r({\delta},{\gamma})<{\sigma}\}$.

So we can assume that $D\setm K$ is infinite.
In this case $x_n=\<{\gamma},\<D,{\sigma}\>,n\>\in \tcal_1$ for
$n\in {\omega}$,
so $c_r({\gamma},{\eta}_{x_n})=0<{\sigma}$ and 
$\{{\eta}_{x_n}:n\in {\omega}\}\in \br D;{\omega};$.

Assume now that 
$\<D,{\sigma}\>\in \ecal$.

If ${\gamma}\in \supp(d)$
then $c_r\restriction[\{{\gamma}\},D]=d\restriction[\{{\gamma}\},D]$  
so we are done.
So we can assume that ${\gamma}\in\supp (c')\setm K$.

If $D\setm K$ is finite then 
${\gamma}'={\varphi}^{-1}({\gamma})\in \supp (c)\subs \supp(d)$
and  $q\in P$ imply that 
the set
\begin{displaymath}
  H=\{{\varepsilon}\in D\cap  K: d({\varepsilon},{\gamma}')<{\sigma}\}
\end{displaymath}
is infinite. But for each ${\varepsilon}\in H$ we have 
$c_r({\varepsilon},{\gamma})=c'({\varepsilon},{\gamma})=
c({\varepsilon},{\gamma}')=d({\varepsilon},{\gamma}')$.
So we can assume that $D\setm K$ is infinite.

In this case 
$x_n=\<{\gamma},\<D,{\sigma}\>,n\>\in \tcal_2$ for
$n\in {\omega}$,
so $c_r({\gamma},{\eta}_{x_n})=0<{\sigma}$ and 
$\{{\eta}_{x_n}:n\in {\omega}\}\in \br D;{\omega};$.

So $r\in P$ and clearly $r\le q,p'$.

Finally for each ${\zeta}<{\rho}$ we have 
${\eta}_{\zeta}\in A$ and $c_r({\xi},{\eta}_{\zeta})={\zeta}$.
So ${\rho}\subs c_r''[\{{\xi}\},A]$.
\end{proof}

\begin{lemma}
$\pcal$ is ${\omega}_2$-c.c.  
\end{lemma}

\begin{proof}[Proof of the lemma]
Since  any family of  conditions of size $\oot$
contains two conditions $p$ and $p'$ which are twins
we can apply the previous lemma to yield that 
$p$ and $p'$ are compatible in $P$.
\end{proof}

Let  $\gcal$ be  the generic filter for $\pcal$
and put $g=\cup\{c:\<c,\acal,{\xi}\>\in\gcal\}$

\begin{lemma}\label{lm:gok}
$g$
establishes $\oot\nar [(\oo;\oot)]^2_{\oo}$
in $V[\gcal]$.  
\end{lemma}

\begin{proof}
Assume that $p\force \dot X=\{\dot{\xi}_{\nu}:{\nu}<\oot\}
\in \br{\kappa};\oot;, \dot Y\in \br{\kappa};\oo;$.

For each ${\rho}<\oo$ we will construct a condition $r\le p$
such that $r\force {\rho}\subs g''[\dot X,\dot Y]$.

Write $p=\<c,\dcal\>$.
For each ${\nu}<\oot$ pick $p_{\nu}=\<c_{\nu},\dcal_{\nu}\>\le p$ such that 
$p_{\nu}\force \dot{\xi}_{\nu}={\xi}_{\nu}$ for some 
${\xi}_{\nu}\in \supp (c_{\nu})$. 
Since CH holds there is $I\in \br \oot;\oot;$ such that 
\begin{enumerate}
  \item $\{\supp (c_{\nu}):{\nu}\in I\}$ forms a $\Delta$-system with kernel
$K$,
  \item for each $\{{\nu},{\mu}\}\in \br I;2;$ 
the conditions $p_{\nu}$ and $p_{\mu}$ are twins.
\end{enumerate}

  Since $P$ satisfies $\oot$-c.c we can assume that 
${\xi}_{\nu}\in \supp (c_{\nu})\setm K$ for ${\nu}\in I$.

Fix ${\mu}\in I$.
Pick a condition $q\le p_{\mu}$, $q=\<d,\ecal\>$,
such that 
$q\force Z\subs \dot Y$
for some $Z\in \br \supp(d)\cap({\kappa}\setm K);{\omega};$. 
Choose ${\nu}\in I$ such that $\supp (c_{\nu})\cap \supp(d)=K$.

By lemma \ref{lm:twins}
there is a condition $r=\<c_r,\dcal_{\nu}\cup \ecal\>\in P$ such that 
$r\le q,p_{\nu}$ and ${\rho}\subs c_r''[\{{\xi}_{\mu}\},Z]$.
Then $r\force {\rho}\subs c_r''[\{{\xi}_{\nu}\},Z]\subs g''[\dot X,\dot Y]$.
\end{proof}

\begin{lemma}
There is no uncountable $g$-rainbow 
set in $V[\gcal]$.
\end{lemma}
\begin{proof}
Indeed, assume that $p_0 \force \dot X\in \br \oot;\oo;$.
Since $\pcal$ is ${\sigma}$-complete there are 
$p\le p_0$, 
$p=\<c,\dcal\>$,
 and  $D\in \br \supp(c);{\omega};$
such that $p\force D\subs \dot X$.
Let
$p'=\<c,
\dcal\cup\{\<D,(\sup \ran (c))+1\>\}.
\>$.
Then $p'\in P$ and $p'\le p$. Moreover 
\begin{displaymath}
p'\force \text{$\dot X$ is not a $g$-rainbow}. 
\end{displaymath}
Indeed, work in $V[\gcal]$, where $p'\in\gcal$. 
Write $X=\{{\xi}_{\nu}:{\nu}\in \oo\}$.
Then for each ${\nu}<{\omega}$ there is 
${\gamma}_{\nu}<\sup \ran (c)+1$ and ${\delta}_{\nu}\in D$ with 
$g({\delta}_{\nu}, {\xi}_{\nu})={\gamma}_{\nu}$.
Then there are ${\nu}<{\mu}<\oo$ with ${\gamma}_{\nu}={\gamma}_{\mu}$.
Thus $g({\delta}_{\nu}, {\xi}_{\nu})={\gamma}_{\nu}={\gamma}_{\mu}=
g({\delta}_{\mu}, {\xi}_{\mu})$ and ${\xi}_{\nu}\ne {\xi}_{\mu}$, 
i.e. $X$ is not a $g$-rainbow.
  \end{proof}

So, by the lemmas  above, $g$ satisfies the requirements of the
theorem.

\end{proof}

\section{$k$-bounded colourings}

\begin{definition}
  Let $X\in \br\oo;\oo;$, $f:\br X;2;\rightarrow \oo$, $k\in {\omega}$.
  \begin{enumerate}[(a)]
\item $f$ is {\em $k$-bounded } iff
{$|f^{-1}\{{\gamma}\}|\le k$} for each ${\gamma}\in \ran (f) $.
\item Put 
  \begin{displaymath}
  \fink X=\{D\in \br{\br X;k;};<{\omega};:d\cap d'=\empt 
\text{ for  each } \{d,d'\}\in \br D;2;\}.    
  \end{displaymath}
\item For $D\in\fink X$ let 
  \begin{displaymath}
    \Homo Df=\{{\alpha}:\forall d\in D\ 
(\forall {\delta},{\delta}'\in d)\ f({\delta},{\alpha})=f({\delta}',{\alpha})
\}.
  \end{displaymath}
\item Given any cardinal ${\mu}$ let
  \begin{displaymath}
    \seqk {\mu}X=\{\<D_i:i<{\mu}\>\subs\fink X: 
(\cup D_i)\cap (\cup
 D_j)=\empt
\text{ for } i<j<{\mu} \}.
  \end{displaymath}
\item 
$f$ is an {\em \ark-function}
iff 
\begin{enumerate}[(i)]
 \item $f$ is $k$-bounded,
 \item  for each $\<D_i:i<{\omega}\>\in \seqk {\omega}X$
there is ${\gamma}<\oo$ such that 
\begin{displaymath}
X\setm {\gamma}\subs\cup\{\Homo {D_i}f:i<{\omega}\}.  
\end{displaymath}
  \end{enumerate}
\end{enumerate}
  \end{definition}

\begin{obs}
An \ark-function $f:\br \oo;2;\to \oo$ establishes
the negative partition relation 
 $\oo \nar^* [({\omega};\oo)]_{{k}-{\rm bdd}}$.
\end{obs}

\begin{proof}
Assume that $A\in \br \oo;{\omega};$ and
$B\in \br \oo;\oo;$. Pick pairwise disjoint sets
$\{d_i:i<{\omega}\}\subs \br A;k;$. Write 
$D_i=\{d_i\}$ and $\vec D=\<D_i:i<{\omega}\>$.
Since $\vec D\in \seqk {\omega}\oo$ and 
$f$ is an \ark-function there is ${\beta}\in B$
such that ${\beta}\in \Homo {D_i}f$ for some $i<{\omega}$,
which means that $|f''[d_i,\{{\beta}\}]|=1$. Since 
$d_i\in \br A;k;$ we are done.   
\end{proof}

\begin{lemma}\label{lm:const}
If CH holds then for each $k\in {\omega}$ 
there is an \ark-function $f:\br\oo;2;\rightarrow \oo$.   
\end{lemma}

\begin{proof}
The construction is standard. 
Let $\{C_{\alpha}:{\omega}\le {\alpha}<\oo\}\subs \br \oo;{\omega};$ be 
disjoint sets.
Fix an enumeration $\<\dvec_{\alpha}:{\omega}\le {\alpha}<\oo\>$
of $\seqk {\omega}{\oo}$ such that 
$\cup \cup \dvec_{\alpha}\subs {\alpha}$.

Let ${\alpha}<\oo$ be fixed.
For each 
${\xi}<{\alpha}$ pick $i_{\xi}\in {\omega}$ such that 
the sets $\{\cup (\dvec_{\xi}(i_{\xi})):{\xi}<{\alpha}\}$ 
are pairwise disjoint.
Choose a function  
$g_{\alpha}:{\alpha}\rightarrow C_{\alpha}$ such that  
\begin{itemize}
\item[($*$)] $g_{\alpha}({\delta})=g_{\alpha}({\delta}')$ iff
 $\{{\delta},{\delta}'\}\in \br d;2;$ for some 
 ${\xi}<{\alpha}$ and $d\in \dvec_{\xi}(i_{\xi})$.  
\end{itemize}

For ${\delta}<{\alpha}$ let $f({\delta},{\alpha})=g_{\alpha}({\delta})$.
\end{proof}

\begin{theorem}\label{tm:pos}
If GCH holds and $f:\br \oo;2;\rightarrow\oo$ is an \ark-function
then there is a c.c.c. poset $P$ such that 
\begin{displaymath}
V^P\models\text{ $f$  c.c.c-indestructibly establishes
 $\oo \nar^* [(\oo;\oo)]_{{k}-{\rm bdd}}$}.
\end{displaymath}
\end{theorem}

Although 
an \ark-function establishes
 $\oo \nar^* [({\omega};\oo)]_{{k}-{\rm bdd}}$
but
there is no function which  c.c.c-indestructibly establishes
 $\oo \nar^* [({\omega};\oo)]_{{k}-{\rm bdd}}$
because Martin's Axiom implies 
$\oo \rightarrow^* (({\omega},\oo))^2_{k-bdd}$.

\begin{theorem}\label{tm:neg}
If GCH holds and $f:\br \oo;2;\rightarrow\oo$ is an \ark-function
then there is a set  $X\in \br \oo;\oo;$ and a c.c.c. poset $Q$ such that 
\begin{displaymath}
V^Q\models\text{ $X$ has a partition into countably many 
$f$-rainbow sets}.  
\end{displaymath}
\end{theorem}

Before proving the theorems above we need to introduce
some notions.

 Given a set x denote $\operatorname{TC}(x)$
 the transitive closure of $x$. 
Let ${\kappa}$ be  a large enough regular cardinals,
(${\kappa}=(2^{\oo})^{+}$ works). Put
$H_{{\kappa}}=
\{x:|\operatorname{TC}(x)|<{\kappa}\}$ and 
$\hcal_{{\kappa}}=\<H_{{\kappa}},\in,\prec\>$, where $\prec$ is a
well-ordering of $H_{\kappa}$.

\begin{definition}
 (a)
 A sequence $\vec N=\<N_{{\alpha}}:{\alpha}\in A\>$ of countable, 
elementary
submodels of $\hcal_{{\kappa}}$ is called an {\em A-chain\/} iff $A\subset\oo$ 
and whenever ${\alpha},{\beta}\in A$ with ${\alpha}<{\beta}$ 
we have $N_{\alpha}\in N_{\beta}$.\\
(b)  Suppose that $\vec N=\<N_{\alpha}:{\alpha}\in A\>$ is an $A$-chain and
$Y\subs\oo$. We say that Y is {\em separated by $\vec N$} iff for each
$C\in\br Y;2;$ there is an ${\alpha}\in A$ with $|N_{\alpha}\cap C|=1$.
  \end{definition}

\begin{lemma}\label{lm:main}
Assume that $f$ is an \ark-function.
If $\<N_m:m\le n\>$ is an elementary $n+1$-chain, $f\in N_0$,
$\dvec_0,\dots, \dvec_{n-1}\in \seqk {\omega}\oo\cap N_0$,  and 
${\alpha}_m\in N_{m+1}\setm N_m$ for $m<n$ then 
the set 
\begin{displaymath}
\{i<{\omega}:\forall m<n\ {\alpha}_m\in \Homo {\dvec_m(i)}f\}
\end{displaymath}
is infinite.
\end{lemma}

\begin{proof}
We prove the lemma by induction on $n$.
So assume that the set  
\begin{displaymath}
I=\{i<{\omega}:\forall m<n-1\ {\alpha}_m\in \Homo {\dvec_m(i)}f\}
\end{displaymath}
is infinite. (If $n=1$ then  $I={\omega}$).

Write $I=\{i_j:j\in {\omega}\}$ and 
for each $\ell<{\omega}$  put 
\begin{displaymath}
\evec^\ell=\<\dvec_{n-1}(i_j): \ell\le j<{\omega}\>. 
\end{displaymath}
Since $f$ is \ark\ and $\evec^\ell\in \seqk {\omega}\oo$, 
there is ${\gamma}_\ell<\oo$
such that 
\begin{displaymath}
\oo\setm {\gamma}_\ell\subs \cup\{\Homo{\dvec_{n-1}(i_j)}{f}:j\in {\omega}\setm \ell\}.  
\end{displaymath}
So if we take ${\gamma}=\sup\{{\gamma}_\ell:\ell<{\omega}\}$ then 
for each ${\alpha}\in \oo\setm {\gamma} $  the set 
\begin{displaymath}
J_{\alpha}=\{i\in I: {\alpha}\in \Homo{\dvec_{n-1}(i)}f\}  
\end{displaymath}
is infinite.

Since $f,\dvec_0,\dots, \dvec_{n-1}, {\alpha}_0,\dots, {\alpha}_{n-2}
\in   N_{n-1}$ we have $I\in N_{n-1}$ and so 
$\evec^\ell\in N_{n-1}$ as well. Thus
$\<{\gamma}_\ell:\ell<{\omega}\>\in N_{n-1}$ and so 
${\gamma}=\sup \<{\gamma}_\ell:\ell<{\omega}\>\in N_{n-1}$ as well.
Hence ${\alpha}_{n-1}\in  N_n\setm N_{n-1}\subs \oo\setm {\gamma}$ and so 
$J_{\alpha_{n-1}}$ is infinite.

But
\begin{displaymath}
J_{\alpha_{n-1}}=
\{i<{\omega}:\forall m<n\ {\alpha}_m\in \Homo {\dvec_m(i)}f\},
\end{displaymath}
so we are done.
\end{proof}

\begin{proof}[Proof of thereon \ref{tm:neg}]
Let $\nvec=\<N_{\xi}:{\xi}<\oo\>$ be an $\oo$-chain with $f\in N_0$  and let
$X\in \br \oo;\oo;$ be $\nvec$-separated.

Define the poset $\qcal=\<Q,\le\>$ as follows:
\begin{displaymath}
Q=\{q\in Fn(X,{\omega}): \text{$q^{-1}\{n\}$ is $f$-rainbow for each 
$n\in \ran q$}\},  
\end{displaymath}
and let $q\le q'$ iff $q\supset q'$.

\begin{lemma}
$\qcal$ satisfies c.c.c.  
\end{lemma}

\begin{proof}[Proof of the lemma]
Assume that $\{q_{\nu}:{\nu}<\oo\}\subs Q$.

Let $x_{\nu}=\dom q_{\nu}$, $L_{\nu}=\ran q_{\nu}$,
and $x_{{\nu},\ell}=q_{\nu}^{-1}\{\ell\}$ for $\ell\in L_{\nu}$.

We can assume that   
\begin{enumerate}[(1)]
\item  $\{x_{\nu}:{\nu}<\oo\}$ forms a
$\Delta$ system with kernel $x$,
\item $x< x_{\zeta}\setm x<x_{\xi}\setm x$  for ${\zeta}<{\xi}<\oo$,
\item $L_{\nu}=L$ for each ${\nu}<\oo$,
\item $q_{\nu}\restriction \br x;2;=q$ for each ${\nu}<\oo$,
\end{enumerate}

For ${\zeta}\in \oo$ let
\begin{displaymath}
F({\zeta})=\{{\xi}<\oo: f''[x_{\zeta},x_{\zeta}\setm x]\cap
f''[x_{\xi},x_{\xi}\setm x]\ne\empt\}.
\end{displaymath}
Since $f$ is $k$-bounded, $F({\zeta})$ is finite, and so 
there is an $F$-free set $Z=\{{\zeta}_i:i<\oo\}\in \br \oo;\oo; $.

For $x\in \oo$ let ${\rho}(x)=\min \{{\nu}:x\in N_{\nu}\}$.
For each ${\xi}\in X$ pick 
$d_{\xi}\in \br \oo;k;$ such that 
 ${\xi}\in d_{\xi}$ and ${\rho}({\eta})={\rho}({\xi})$ 
for each ${\eta}\in d_{\xi}$.

For ${\zeta}\in Z$  let
$D_{{\zeta}}=\{d_{\xi}:{\xi}\in x_{{\zeta}}\setm x\}$.

Let $\dvec=\<D_{{\zeta}_i}:i<{\omega}\>$.
Clearly $\dvec\in \seqk {\omega}{\oo}$.

Since CH holds there is ${\gamma}<\oo$ such that 
$\dvec\in N_{\gamma}$.

Let ${\zeta}\in Z$ such that $N_{\gamma}\cap (x_{\zeta}\setm x)=\empt$.

Apply lemma \ref{lm:main} for $n=|x_{\zeta}\setm x|$,
$\dvec_m=\dvec$ for $m<n$ and 
$\{{\alpha}_m:m<n\}=x_{\zeta}\setm x$.
Then,  there is $i<{\omega}$ such that 
\begin{displaymath}
(\forall m<n) \  {\alpha}_m\in \Homo {\dvec(i)}f.
\end{displaymath}
By the construction it means that 
\begin{displaymath}
(\forall {\eta}\in x_{{\zeta}}\setm x)\ 
(\forall {\xi}\in x_{{\zeta}_i}\setm x)\ 
(\forall {\delta}\in d_{\xi})\ f({\delta},{\eta})=f({\xi},{\eta}). 
\end{displaymath}

\smallskip

\noindent
{\bf Claim:} $q_{{\zeta}_i}\cup q_{\zeta}\in Q$, i.e. 
$f$ is 1--1 on $\br x_{{\zeta}_i,\ell}\cup x_{{\zeta},\ell};2;$
for all $\ell\in L$.

\smallskip

Let  
${\xi},{\eta},{\xi}',{\eta}'\in x_{{\zeta}_i,\ell}\cup x_{{\zeta},\ell}$ with 
${\xi}<{\eta}$ and ${\xi}'<{\eta}'$ such that 
$f({\xi},{\eta})=f({\xi}',{\eta}')$.

Assume first that $\{{\xi},{\eta}\},\{{\xi}',{\eta}'\}\in \br 
x_{{\zeta}_i,\ell};2;\cup \br x_{\zeta,\ell};2;$.
Since $q_{{\zeta}_i}, q_{\zeta}\in Q$ we can assume that 
$\{{\xi},{\eta}\}\in\br x_{{\zeta}_i,\ell};2;\setm  \br x_{\zeta,\ell};2;$ and 
$\{{\xi}',{\eta}'\}\in \br x_{\zeta,\ell};2;\setm \br x_{{\zeta}_i,\ell};2;$,
(or $f({\xi},{\eta})=f({\xi}',{\eta}')$ implies 
$\{{\xi},{\eta}\}=\{{\xi}',{\eta}'\}$).
Then $f({\xi},{\eta})\in f''[x_{{\zeta}_i}, x_{{\zeta}_i}\setm x]$ and 
$f({\xi}',{\eta}')\in f''[x_{\zeta}, x_{\zeta}\setm x]$, 
so ${\zeta}_i\notin F({\zeta})$ implies $f({\xi},{\eta})\ne f({\xi}',{\eta}')$.

So we can assume that e.g. $\{{\xi},{\eta}\}\notin \br 
x_{{\zeta}_i,\ell};2;\cup \br x_{\zeta,\ell};2;$, 
i.e. ${\xi}\in x_{{\zeta}_i,\ell}\setm x$ and ${\eta}\in
x_{\zeta,\ell}\setm x$. But we know that 
\begin{displaymath}
(\forall {\delta}\in d_{\xi})\ f({\delta},{\eta})=f({\xi},{\eta}).  
\end{displaymath}
 Since $f$ is $k$-bounded and
$|d_{\xi}|=k$
we have 
\begin{displaymath}
\bigl\{\{{\xi}',{\eta}'\}:f({\xi}',{\eta}')=f({\xi},{\eta})
\bigr\}=
\bigl\{\{{\delta},{\eta}\}:{\delta}\in d_{\xi}\bigr\}.  
\end{displaymath}
 But 
$d_{\xi}\cap (x_{{\zeta}_i,\ell}\cup x_{\zeta,\ell})=\{{\xi}\}$ because 
${\rho}({\delta})={\rho}({\xi})$ for each ${\delta}\in d_{\xi}$.
Hence 
$f({\xi}',{\eta}')=f({\xi},{\eta})$ implies
${\xi}={\xi}'$ and ${\eta}={\eta}'$.
\end{proof}

Since $\{q\in Q:{\xi}\in \dom q\}$ is dense in $\qcal$
for each ${\xi}\in X$ we have that if
$\gcal$ is the  generic filter in $Q$ and 
$g=\cup\gcal$, then 
$\{g^{-1}\{n\}:n\in {\omega}\}$
is a partition of $X$ into countably many
$f$-rainbow sets,
which completes the proof of Theorem \ref{tm:neg}. 
\end{proof}

To prove theorem \ref{tm:pos} we need some more preparation.
We will use a  black box theorem from \cite{So}.

Given a set $K$ and a natural number $m$ let 
\begin{displaymath}
\text{{$\finmok$}=
$\{s:\text{$s$ is a function,}\ \dom (s)\in \br \oo;m;,\ran (s)\subs K\}$}.
\end{displaymath}
A sequence 
$\<s_{\alpha}:{\alpha}<\oo\>\subs\finmok$ is 
 {\em \domd } iff {$\domm[ s_{\alpha}]\cap \domm[ s_{\beta}]=\empt$} 
all ${\alpha}<{\beta}<\oo$. 

 Let $H$ be a graph on $\oo\times K$, $m\in {\omega}$.
{We say that $H$ is {$m$-solid}
if given any {\domd\ sequence $\<s_{\alpha}:{\alpha}<\oo\>\subs \finmok $}}
{there are ${\alpha}<{\beta}< \oo$ such that 
{\begin{displaymath}
[s_{\alpha},s_{\beta}]\subs H.
\end{displaymath}}}
{$H$ is called {\em strongly solid} iff it is $m$-solid for each $m\in{\omega}$.}

\begin{btheorem}[{\cite[Theorem 2.2]{So}}]
Assume $2^{\oo}={\omega}_2$. If $H$ is a {strongly solid} graph on
$\oo\times K$, where $|K|\le 2^{\oo}$, then for each
$m\in {\omega}$ there is a c.c.c poset $P$ of size ${\omega}_2$
such that 
\begin{displaymath}
V^P\models
\text{\em ``$H$ is {c.c.c-indestructibly $m$-solid.}''}
\end{displaymath}
\end{btheorem}

The theorem above is build on a method of 
 Abraham and  Todor\v cevi\v c from \cite{AT}.

We need one more lemma before we can apply
the Black Box Theorem above.

\begin{lemma}\label{lm:disjoint}
There is a function $r:\oo \rightarrow {\omega}$ such that 
for each $A,B\in \br \oo;<{\omega};$ if 
$r(A)=r(B)$ then $A\cap B$ is an initial segment of $A$ and $B$.  
\end{lemma}

\begin{proof}
Let $\dcal$ be a countable dense subset of the product space
${\omega}^{\oo}$. Moreover, for each ${\alpha}<\oo$ fix a function
$f_{\alpha}:{\alpha}\stackrel{1-1}{\rightarrow}{\omega}$.

Let $A=\{{\alpha}_0,\dots, {\alpha}_{n-1}\}\in \br \oo;<{\omega};$,
${\alpha}_0<\dots {\alpha}_{n-1}$.

Pick $d_A\in \dcal$ such that 
$d_A({\alpha}_i)=i$ for each $i<|A|$.
Let 
\begin{displaymath}
r(A)=\< d_A, \<f_{{\alpha}_i}''(A\cap {{\alpha}_i}):i<|A|\>\>. 
\end{displaymath}
Since the range of $r$ is countable it is enough to prove that
if $r(A)=r(B)$ then $A\cap B$ is an initial segment of $A$ and $B$.  

Write $A=\{{\alpha}_i:i<m\}$, 
${\alpha}_0<\dots <{\alpha}_{n-1}$,
 and $B=\{{\beta}_j:j<m\}$, ${\beta}_0<\dots {\beta}_{m-1}$.

Assume that ${\alpha}_i={\beta}_j$.
Then $d_A({\alpha}_i)=i$ and $d_B({\beta}_j)=j$.
Since $d_A=d_B$ it follows that $i=j$.
So $r(A)=r(B)$ yields 
$f_{{\alpha}_i}''(A\cap {\alpha}_i)=f_{{\alpha}_i}''(B\cap {\alpha}_i)$.
Since  $f_{{\alpha}_i}$ is 1--1 on ${\alpha}_i$ it follows
that $A\cap {\alpha}_i=B\cap {\alpha}_i$.
\end{proof}

We will use the following corollary of this lemma.
\begin{corollary}\label{cor:r}
There is a function $r:\oo \to {\omega}$ such that 
for each $A,B\in \br \oo;<{\omega};$ if $\min(A)\ne \min(B)$ and  
$r(A)=r(B)$ then $A\cap B=\empt$. 
\end{corollary}

\begin{proof}[Proof of Theorem \ref{tm:pos}]
Let $\nvec=\<N_{\xi}:{\xi}<\oo\>$ be an $\oo$-chain
with $f\in N_0$. Fix the function $r$
from corollary \ref{cor:r} above.

For ${\xi}\in \oo$ let ${\rho}({\xi})=\min \{{\nu}:{\xi}\in N_{\nu}\}$.

Let $K=\br \oo;k;\times \oo \times {\omega}$.
For any function $c:\oo \to \oo$ define a
graph $H_c$ on  $\oo\times K$ as follows.
 
If $x,x'\in \oo\times K$, $x=\<{\zeta},\<d,{\xi},m\>\>$, 
$x'=\<{\zeta}',\<d',{\xi}',m'\>\>$, 
${\zeta}<{\zeta}'$, let  $\{x,x'\}$ be
an edge in $H_c$ provided \\
IF
\begin{enumerate}[(1)]
\item $m=m'$,
\item ${\rho}({\xi}')={\zeta}' $,
\item ${\zeta}<\min d$,
\item $r(\{{\zeta}\}\cup d)=m$,
\end{enumerate}
THEN
\begin{enumerate}[(1)]
\addtocounter{enumi}{4}
\item $c({\delta},{\xi}')=c({\varepsilon},{\xi}')$ 
for each ${\delta},{\varepsilon}\in d$.
\end{enumerate}

\begin{lemma}\label{lm:1}
If  $H_c$ is $1$-solid for some  colouring $c$ then $c$ establishes
$\oo\nar^*[(\oo;\oo)]_{k-b.d.d.} $.   
\end{lemma}
 
\begin{proof}

We will show that for all $X=\{{\xi}_{\beta}:{\beta}<\oo\}\in \br \oo;\oo;$
and for all disjoint family  $\{d_{\alpha}:{\alpha}<\oo\}\subs \br\oo;k;$
there are  ${\alpha},{\beta}<\oo $ such  that
$\max d_{\alpha}<{\xi}_{\beta}$ and 
$|c''[d_{\alpha}, \{{\xi}_{\beta}\}]|=1$. 

By thinning out and renumerating of the sequences 
we can assume that
\begin{enumerate}[(1)]
\item  
${\rho}({\xi}_{\alpha})<\min {\rho}''d_{\alpha}<\max
{\rho}''d_{\alpha}<{\rho}({\xi}_{{\alpha}+1})$ 
for ${\alpha}<{\beta}<\oo$,
\item $r(\{{\rho}({\xi}_{\alpha})\cup d_{\alpha}\})=m$ for some 
$m\in {\omega}$ for each ${\alpha}\in \oo$.
\end{enumerate}

Let 
$x_{\alpha}=\<{\rho}({\xi}_{\alpha}),\<d_{\alpha},{\xi}_{\alpha},
  m\>\>$
for ${\alpha}<\oo$. Since 
the sequence $\<\{x_{\alpha}\}:{\alpha}<\oo\>$ is \domd,
and (1)-(4) hold for each ${\alpha}<{\beta}<\oo$, there are
${\alpha}<{\beta}<\oo$ such that (5) holds for 
$x_{\alpha}$ and $x_{\beta}$ because $H_c$ is $1$-solid,
i.e.  $|c''[d_{\alpha},\{{\xi}_{\beta}\}]|=1$,
 which was to be proved.  
\end{proof}

\begin{lemma}\label{lm:strongly}
If $c$ is an \ark-function and $CH$ holds 
then $H_{c}$ is strongly solid.  
\end{lemma}

\begin{proof}
Let $m\in {\omega}$ and 
$\<E_{\alpha}:{\alpha}<\oo\>\subs \finmok$ be a \domd\ sequence.

Write $E_{\alpha}=\{x_{{\alpha},i}:i<m\}$,
$x_{{\alpha},i}=
\<{\zeta}_{{\alpha},i}, \<d_{{\alpha},i},{\xi}_{{\alpha},i}, 
n_{{\alpha},i}\>\>$.

We can assume that 
\begin{enumerate}[(i)]
\item $n_{{\alpha},i}=n_i$,
\item ${\rho}({\xi}_{{\alpha},i})={\zeta}_{{\alpha},i}$
\item ${\zeta}_{{\alpha},i}<\min d_{{\alpha},i}$,
\item $r(\{{\zeta}_{{\alpha},i}\}\cup(d_{{\alpha},i}))=n_i$,  
\item $\max {\rho}''d_{{\alpha},i}<{\zeta}_{{\beta},j}$ for 
${\alpha}<{\beta}<\oo$ and $i,j<m$.
\end{enumerate}

Let $N=\{n_i:i<m\}$.
For  ${\alpha}<\oo$ and $n\in N$ put 
$D_{{\alpha},n}=\{d_{{\alpha},i}:n_i=n\}$.

\noindent{\bf Claim:} $D_{{\alpha},n}\in \fink \oo$.

Indeed, if $i\ne j<m$ and $n_i=n_j$ then 
$r(\{{\zeta}_{{\alpha},i}\}\cup d_{{\alpha},i})=n_i=n_j=
r(\{{\zeta}_{{\alpha},j}\}\cup d_{{\alpha},j})$ but 
$\min (\{{\zeta}_{{\alpha},i}\}\cup d_{{\alpha},i})=
{\zeta}_{{\alpha},i}\ne {\zeta}_{{\alpha},j}=
\min (\{{\zeta}_{{\alpha},j}\}\cup d_{{\alpha},j})$
so $d_{{\alpha},i}\cap d_{{\alpha},j}=\empt$
by the choice of the function $r$.

\medskip
 
(iii) and (v) together give $\max (\cup D_{{\alpha},n})<\min
(\cup D_{{\beta},n})$
for ${\alpha}<{\beta}<\oo$ and $n\in N$.

Thus $\dvec'_n=\<D_{\ell,n}:\ell<{\omega}\>\in \seqk {\omega}\oo$.

Since CH holds there is ${\gamma}<\oo$ such that 
$\{\dvec'_n:n\in N\}\subs N_{\gamma}$. 
Pick ${\alpha}<\oo$ such that 
$N_{\gamma}\cap \{{\zeta}_{{\alpha},j}:j<m\}=\empt$.

Let $\dvec_j=\dvec'_{n_j}$ for $j<m$.

We are going to apply lemma \ref{lm:main} as follows:
 $\mvec=\<N_{\gamma},N_{{\zeta}_j}:j<m\>$ is an elementary
$m+1$-chain, $f,\dvec_0,\dots, \dvec_{m-1} \in N_0$
and ${\xi}_{\alpha,j}\in N_{{\zeta}_j}\setm N_{{\zeta}_{j-1}}$ for
$j<m$, where
${\zeta}_{-1}={\gamma}$. Hence, by lemma \ref{lm:main}
there is $\ell<{\omega}$ such that 
for each $j<m$
\begin{displaymath}\tag{$\circ$}\label{circ}
{\xi}_{{\alpha},j}\in \Homo {\vec D_j(\ell)}f.  
\end{displaymath}

\noindent{\bf Claim} $[x_\ell,x_{\alpha}]\subs H_c$.

Let $i,j<m$. We show $\{x_{\ell,i}, x_{{\alpha},j}\}\in H_c$.
(2)-(4) holds by the construction.
If $n_i\ne n_j$ then (1) fails so we are done.
Assume that $n_i=n_j=n\in N$.
Then $d_{\ell,i}\in \vec D'_n(\ell)=\vec D_j(\ell)$.
Thus 
\begin{displaymath}
(\forall {\delta},{\delta}'\in d_{\ell,i})\ 
  f({\delta},{\xi}_{{\alpha},j})
=f({\delta}',{\xi}_{{\alpha},j})  
\end{displaymath}
by (\ref{circ}).
Hence (5) holds and so  $\{x_{\ell,i}, x_{{\alpha},j}\}\in H_c$.
\end{proof}

Now we can easily conclude the proof of  \ref{tm:pos}.

Let $f:\br \oo;2;\rightarrow\oo$ be an \ark-function.
By lemma \ref{lm:strongly}, the graph $H_f$
is strongly solid.
Since $GCH$ holds, we can apply our Black Box Theorem
to find a c.c.c. poset $P$
such that 
\begin{displaymath}
V^P\models\text{$H_f$ is c.c.c-indestructibly 1-solid.}  
\end{displaymath}
But then, by lemma \ref{lm:1},
\begin{displaymath}
V^P\models\text{$f$  c.c.c-indestructibly establishes
$\oo\nar^*[(\oo;\oo)]_{k-b.d.d.} $.}  
\end{displaymath}
\end{proof}

\begin{proof}[Proof of theorem \ref{tm:acgen}]
Since GCH holds,  by lemma \ref{lm:const} there is an \ark-function 
$g:\br\oo;2;\to \oo$ .
By theorem \ref{tm:neg} there is a set $X\in \br \oo;\oo;$
and a c.c.c. poset $Q$ such that 
\begin{displaymath}
V^Q\models\text{ $X$ has a partition into countably many 
$g$-rainbow sets}.  
\end{displaymath}
Let $h:\oo\to X$ be a bijection and 
put $f=g \circ h$.
Then   
\begin{displaymath}
V^Q\models\text{ $\oo$ has a partition into countably many 
$f$-rainbow sets}.  
\end{displaymath}
Since $f$ is an \ark-function as well, we can apply
theorem \ref{tm:pos} to obtain that 
\begin{displaymath}
V^P\models\text{ $f$  c.c.c-indestructibly establishes
 $\oo \nar^* [(\oo;\oo)]_{{k}-{\rm bdd}}$},
\end{displaymath}
fro some c.c.c. poset $P$,
which proves the theorem.
\end{proof}
Lemma \ref{lm:const} and theorem \ref{tm:pos}  give immediately
\begin{corollary}\label{tm:main}
 $\oo \nar^* [(\oo;\oo)]_{{k}-{\rm bdd}}$ is consistent with 
Martin's Axiom.
\end{corollary}

\end{document}